\documentclass[12pt,a4paper]{article}
\usepackage{pifont}
\usepackage[latin1]{inputenc}
\usepackage{amsmath}
\usepackage{amsthm}
\usepackage{graphicx}
\usepackage{color}
\usepackage[margin= 1 in]{geometry}
\usepackage{amssymb}
\usepackage{mathrsfs}
\definecolor{gruen}{cmyk}{1.0,0.2,0.7,0.07}
\definecolor{mag}{cmyk}{0.0,0.9,0.3,0.0}

\theoremstyle{plain}
\newtheorem{theorem}{Theorem}[section]

\newtheorem{lemma}[theorem]{Lemma}
\newtheorem{corollary}[theorem]{Corollary}
\newtheorem{conjecture}[theorem]{Conjecture}

\newtheorem{observation}[theorem]{Observation}

\theoremstyle{definition} 

\newtheorem{definition}[theorem]{Definition}

\begin{document}

\title{Completing partial Latin squares with two filled rows and three
filled columns\footnote{This paper is based on the Bachelor thesis 
\cite{Goransson} by G\"oransson
written under the supervision of Casselgren.}}

\author{
{\sl Carl Johan Casselgren}\thanks{{\it E-mail address:} 
carl.johan.casselgren@liu.se. \, Casselgren was supported by a grant from the
Swedish Research Council (2017-05077)}\\ 
Department of Mathematics \\
Link\"oping University \\ 
SE-581 83 Link\"oping, Sweden
\and
{\sl Herman G\"oransson}\thanks{{\it E-mail address:} 
 herman.goransson@gmail.com}\\ 
Department of Mathematics \\
Link\"oping University \\ 
SE-581 83 Link\"oping, Sweden
}

\maketitle

\begin{abstract}
	Consider a partial Latin square
	$P$ where the  first two  rows and first three columns
	are completely filled, and every other cell of $P$ is empty.
	It has been conjectured that all such partial Latin squares
	of order at least $8$ are completable.
	Based on a technique by Kuhl and McGinn we describe a framework
	for completing partial Latin squares in this class.
	Moreover,
	we use our method for proving that all  partial Latin squares
	from this family,
	where the intersection of the nonempty rows and columns
	form a Latin rectangle with three distinct symbols, is completable.
\end{abstract}


\section{Introduction}
	Consider an $n \times n$ array $P$
	where each cell contains at most one symbol from $[n]=\{1, \dots,n \}$.
	$P$ is called a \emph{partial Latin square}
	if each symbol occurs at most once in every row and column.
	If no cell in $P$ is empty, then it is a \emph{Latin square}.
	An $r \times s$ array with entries
	from $\{1,\dots,n\}$, where $n = \max\{r,s\}$,
	is called a \emph{Latin rectangle} if each symbol occurs
	at most once in every row and column, and no cell is empty.

	The cell in position $(i,j)$ in an array $A$ is denoted by $(i,j)_A$,
	and the symbol in cell $(i,j)_A$ is denoted by $A(i,j)$;
	if $A(i,j)=k$, then $k$ is an \emph{entry}
	of cell $(i,j)_A$; 
	we write $A(i,j) = \emptyset$ if $(i,j)_A$ is empty.

	An $n \times n$ Latin square $L$ is a \emph{completion}
	of an $n \times n$ partial Latin square $P$ if $L(i,j) = P(i,j)$ 
	for each nonempty cell $(i,j)_P$ of $P$.
	$P$ is \emph{completable} if there is such a Latin square;
	otherwise, $P$ is {\em non-completable}.
	The problem of completing partial Latin squares
	is a classic within combinatorics and 
	several families of partial Latin squares have been proved
	to admit completions. Let us here
	just mention a few classic and recent results.

	In general, it is an $NP$-complete problem to determine
	if a partial Latin square is completable \cite{Colbourn}.
	Thus it is natural to ask for completability of particular families of
	partial Latin squares. 
	A classic result due to Ryser \cite{Ryser} 
	states that if $n \geq r,s$, then
	every $n \times n$ partial Latin square whose nonempty cells form
	an $r \times s$ subrectangle $Q$ is completable if and only if each
	of the symbols $1,\dots,n$ occurs at least $r+s-n$ times in $Q$.
	Another classic result
	is Smetaniuk's proof \cite{Smetaniuk}
	of Evans' conjecture \cite{Evans} that every 
	$n \times n$ partial Latin square
	with at most $n-1$ entries is completable.
	This was also independently proved by Andersen
	and Hilton \cite{AndersenHilton}.
	
	Adams, Bryant and Buchanan \cite{AdamsBryantBuchanan} characterized which
	partial Latin squares with 
	$2$ completely filled rows and columns, 
	and where all other cells are empty, are completable,
	and by results of Casselgren and H\"aggkvist \cite{CasselgrenHaggkvist},
	and Kuhl and Schroeder \cite{KuhlSchroeder},
	all partial Latin squares of order at least $6$
	with all entries in one fixed column or row, or containing a
	prescribed symbol, are completable.

	The result 
	that all partial Latin squares with two filled rows
	and two filled columns of order at least $6$ are completable
	was first proved in Buchanan's
	PhD thesis \cite{Buchanan}; the shortened version in \cite{AdamsBryantBuchanan}
	is still over 25 pages long and also relies on a computer
	search for verifying completability for small orders.
	Quite recently, Kuhl and McGinn \cite{KuhlMcGinn} gave a short proof 
	of this result based on Smetaniuk's aforementioned proof
	of the famous Evans'
	conjecture.
	They also presented a conjecture on completing partial Latin squares
	with two filled rows and any number of filled columns.
	For the case of three filled columns their conjecture
	reads as follows.

	\begin{conjecture}
	\label{conj:general}
		Every partial Latin square of order at least $8$
		with two completely filled rows and three completely filled columns,
		and where all other cells are empty, is completable.
	\end{conjecture}

		The non-completable partial Latin squares in Figure \ref{fig}
	show that the condition
	$n \geq 8$ in Conjecture \ref{conj:general} is necessary.

\begin{figure}[h]
\label{fig}
\begin{center}
      \begin{tabular}{c c c}
      \begin{tabular}{|c|c|c|c|c|}
	\hline	1 & 2 & 3  & 4 & 5 \\
	\hline	2 & 4 & 5 & 3 & 1  \\
	\hline	3 & 5 & 1 & & \\
	\hline	4 & 3 & 2 & & \\
	\hline	5 & 1 & 4 & & \\
	\hline
    
    \end{tabular} 
       & \quad
       \begin{tabular}{|c|c|c|c|c|c|}
	\hline	1 & 2 & 3 & 4 & 5 & 6  \\
	\hline	2 & 6 & 1 & 5 & 4 & 3 \\
	\hline	3 & 5 & 4 &  & &\\
	\hline  4 & 3 & 5 &  & &\\
	\hline  5 & 4 & 6 &  & &\\
	\hline  6 & 1 & 2 &  & &\\
	\hline
    
    \end{tabular} 
        & \quad
      \begin{tabular}{|c|c|c|c|c|c|c|}
	\hline	1 & 2 & 3 & 4 & 5 & 6 & 7  \\
	\hline	2 & 1 & 7 & 6 & 4 & 5 & 3 \\
	\hline	3 & 7 & 2 &  & & &\\
	\hline  4 & 5 & 6 &  & & & \\
	\hline  5 & 6 & 4 &  & & & \\
	\hline  6 & 4 & 5 &  & & & \\
	\hline  7 & 3 & 1 &  & & & \\
	\hline
    
    \end{tabular}

       \end{tabular}
       \end{center}
    \caption{Non-completable partial Latin squares of order $3, 4$ and $5$.}
\end{figure}

	In this paper, we take the first step towards settling Conjecture
	\ref{conj:general} by
	proving it in the special case
	when the intersection of the nonempty rows and columns form 
	a Latin rectangle of order $3$; that is, it contains only three
	distinct symbols.
	
	Our proof of this result
	employs methods from \cite{KuhlMcGinn};
	in fact,
	based on the techniques from that paper we shall present a
	general framework for completing partial Latin squares
	with two filled rows and three filled columns. We then use
	this framework for giving a short proof of the fact that
	all such partial Latin squares where the intersection of the
	filled rows and columns form a Latin rectangle of order $3$
	are completable.

	In Section 2 we review some material from \cite{KuhlMcGinn} 
	and introduce some additional tools,
	and
	in Section 3 we present our method for completing partial Latin
	squares with two filled rows and three filled columns
	and prove a special case of Conjecture \ref{conj:general}.

	\section{Preliminaries}
	
	Two partial Latin squares $P$ and $P'$ are {\em isotopic}
	if $P'$ can be obtained from $P$ by permuting
	rows, permuting columns and/or permuting symbols in $P$. Note that
	if $P$ and $P'$ are isotopic, then $P$ is completable
	if and only if $P'$ is completable.
	
	A partial Latin square $P$ of order $n$ can  equivalently be described
	as a subset
	of $[n] \times [n] \times [n]$, where $(r,c,s) \in P$ if and only if
	$s=P(r,c)$. We shall swap freely between this representation and
	the array representation of partial Latin squares.
	
	A {\em conjugate} of $P$ is an array in which the coordinates
	of each triple $(r,c,s)$
	of $P$ are uniformly permuted according to one of the following  six 
	ways: $$(r,c,s), (c,r,s), (s,c,r), (c,s,r), (r,s,c), (s,r,c).$$
	If $P$ is a partial Latin square, then any conjugate of $P$ is a partial
	Latin square as well. Moreover, any conjugate of $P$ is completable if and only if
	$P$ is.
		
	An \emph{intercalate} in an $n \times n$ partial Latin square $L$ is a set 
	$$C =\{ (r_1,c_1)_L, (r_1,c_2)_L, (r_2,c_1)_L, (r_2,c_2)_L \}$$
	of cells in $L$ such that $$L(r_1,c_1)=L(r_2,c_2)=s_1 \text{ and }
	L(r_1,c_2)=L(r_2,c_1)=s_2.$$
	A \emph{swap on $C$}
 	is the operation $L \mapsto L'$, where $L'$ is an $n \times n$ partial
	Latin square with 
	$$L'(r_1,c_1)=L'(r_2,c_2)=s_2, \, L'(r_1,c_2)=L'(r_2,c_1)=s_1,$$
	and $L'(i,j)=L(i,j)$ for all other $(i,j)$.
	
	We shall need the following well-known theorem 
	first proved by M. Hall \cite{MHall}.
	
	\begin{theorem}
	\label{th:MHall}
			Every partial Latin square of order $n$ with 
			$r \leq n$ completely filled columns and no other filled cells
			is completable.
	\end{theorem}

	We shall need some further auxiliary results; the following
	lemma is a simple
	consequence of Hall's condition for matchings in bipartite graphs.
	Denote by $\delta(G)$ the minimum degree of a graph $G$.
	
	\begin{lemma}
	\label{lem:Hall}
		If $B$ is a balanced bipartite graph with parts $V_1$ and $V_2$,
		and $\delta(B) \geq \frac{|V_1|}{2}$, then $B$ has a perfect matching.
	\end{lemma}
	
	This lemma enables us to prove the following.
	
	\begin{lemma}
	\label{lem:PLS}
		Let $P$ be an $n \times n$ partial Latin square with $r$ completely
		filled columns, one partially filled column with $s$ filled
		cells and where all other columns are empty. If $n \geq 2r +s$,
		then $P$ is completable.
	\end{lemma}
	\begin{proof}
		Without loss of generality, we assume that the cells
		in rows $1,\dots, n-s$ of the partially filled column $c$ of $P$ are empty,
		and	that symbols $1,\dots, n-s$
		do not appear in column $c$ of $P$.
	
		We form a bipartite graph $B$ with parts 
		$V_1= \{r_1, r_2, \dots, r_{n-s}\}$
		and $V_2 = \{1,\dots, n-s\}$, and where $r_i j \in E(B)$ if and only
		if symbol $j$ does not appear in row $i$ of $P$.
		Now, $d_B(r_i) \geq n-s-r$, and $d_B(j) \geq n-s-r$, since there are at most
		$r$ different symbols in each of the $n-s$ first rows of $P$, 
		and each of the symbols $1,\dots, n-s$ appears in at most
		$r$ different rows. Thus $\delta(B) \geq n-s-r \geq \frac{n-s}{2}$, by
		assumption; so by Lemma \ref{lem:Hall}, $B$ contains a perfect matching $M$.
		Now, for each empty cell $(i,c)_P$ in column $c$ of $P$
		we assign the symbol $j$ satisfying that $r_i j \in M$ to $(i,c)_P$; the obtained
		partial Latin square $P'$ 
		has $r+1$ completetely filled columns and all other cells of $P'$
		are empty. 
		Thus, by Theorem \ref{th:MHall}, $P'$ is completable, and so,
		$P$ has a completion.		
	\end{proof}
	
	Finally, we shall need the following result
	proved by H\"aggkvist; see e.g. \cite{Asratian}.
	We denote by $\mathrm{PLS}(a,b;n)$ the set of all $n \times n$ partial Latin
	squares
	with $a$ completely filled rows and $b$ completely filled
	columns, and where all other cells are empty.
	
	\begin{theorem}
	\label{th:Haggkvist}
		If $P \in \mathrm{PLS}(b,b;n)$ is a partial Latin square
		where the cells in the intersection of the filled rows and columns
		form a Latin square, then $P$ is completable.
	\end{theorem}

	\subsection{Smetaniuk completion}
	
		A main ingredient in Smetaniuk's resolution
		of the Evans' conjecture is what we call the {\em Smetaniuk completion}
		of a partial Latin square.
		Below we briefly review this technique along with its generalization
		by Kuhl and McGinn \cite{KuhlMcGinn}.

	If $P$ is a partial Latin square of order $n$, then the set
	$D = \{(i,i)_P, i \in [n]\}$ is called the {\em forward diagonal}
	of $P$. A cell $(r,c)_P$ of $P$ lies {\em below} $D$ if $c <r$;
	the cell is {\em above} $D$ if is neither below $D$, nor in $D$.
	
	For a partial Latin square $P$ of order $n$, we define a new partial
	Latin square $T(P)$ of order $n+1$ by setting 
	$$T(P) = \{(r+1,c,s) : (r,c,s) \in P, c < r\} \cup 
	\{(i,i,n+1) : i \in [n+1]\}.$$
	Note that all cells above the forward diagonal of $T(P)$ are empty.
	
	\begin{theorem} (Smetaniuk completion \cite{Smetaniuk})
	\label{th:Smet}
		If $P$ is a completable partial Latin square, then $T(P)$ is completable.
	\end{theorem}
	
	In \cite{KuhlMcGinn}, the authors generalize 
	the above ideas as follows.
	Let $P$ be a partial Latin square of order $n$.
	If $n$ is odd, then the {\em forward augmented diagonal} $D^2$ of $P$
	is defined as the set $$D^2=\{(i,i)_P :
	i \in \{4, 6, 8, \dots,n-1\}\} \cup 
	\{(1,1)_P, (2,1)_P, (3,2)_P, (3,3)_P\};$$
	if $n$ is even, then the forward augmented diagonal is defined as the set
	$$D^2= \{(i,i)_P, (i,i+1)_P, (i+1,i)_P, (i+1,i+1)_P : i \in\{1,3,5,\dots,n-1\}\}.$$
	The properties for a cell of lying below or above the augmented forward
	diagonal is defined analogously as above.
	
	For a partial Latin square $P$ of order $n$  we define a partial Latin
	square $T^2(P)$ of order $n+2$,
	with augmented forward diagonal
	$D^2$,
	by setting
	\begin{itemize}
		
		\item[(i)] $T^2(P)(i,j) = P(i-2,j)$, if $(i,j)_{T^2(P)}$ 
		lies below $D^2$ of $T^2(P)$,
	
		\item[(ii)] $T^2(P)(i,j) \in \{n+1, n+2\}$ if $(i,j)_{T^2(P)} \in D^2$, and
		
		\item[(iii)] the cells of $T^2(P)$ above $D^2$ are empty.
	
	\end{itemize}
	Note that the augmented forward diagonal of $T^2(P)$ is uniquely defined
	up to switching symbols on subarrays of $D^2$. As in \cite{KuhlMcGinn},
	since this suffices for our purposes, we shall be content with this
	definition. Moreover,	
	in \cite{KuhlMcGinn} the authors worked with the (augmented) back diagonal
	rather than the (augmented) forward diagonal. By isotopy, this makes no
	difference for the purpose of completability; thus, since
	the augmented forward diagonal is better suited for our purposes, we reformulate
	the results of \cite{KuhlMcGinn} to this setting. Hence, by isotopy,
	we have the following.
	
	\begin{theorem} 
	\cite{KuhlMcGinn}
	\label{th:T2}
		If $P$ is a completable partial Latin square of order $n$,
		then $T^2(P)$ is completable
		(for any choice of the augmented forward diagonal satisfying (ii)).
	\end{theorem}
	
	The proof of this theorem in \cite{KuhlMcGinn} yields a Latin square
	which we shall refer to as the {\em Smetaniuk completion of $T^2(P)$}.
	Furthermore, when applying this theorem below, the augmented
	forward diagonal in the considered partial Latin squares will
	generally contain symbols $1$ and $2$; again, by isotopy, this of course
	makes no difference for the purpose of completability.

	\begin{observation} \cite{KuhlMcGinn}
	\label{obs1}
		Let $P$ be a Latin square of order $n$ and let $L$ be the Smetaniuk
		completion of $T^2(P)$ with augmented forward diagonal $D^2$.
		Then the following holds:
		
		\begin{itemize}
		
			\item[(i)] $L(i,j) = P(i-2,j)$ if cell $(i,j)$ is below $D^2$ of
			$L$.
			
			\item[(ii)] $L(i,j) \in \{n+1, n+2\}$ if $(i,j) \in D^2$.
			
			\item[(iii)] For odd $n$, if $\{P(1,2), P(1,3)\} \cap 
			\{P(2,4), P(2,5), P(3,4), P(3,5)\} = \emptyset$,
			then $L(3,4) = P(1,4)$ and $L(3,5) = P(1,5)$.
					
		\end{itemize}
			\end{observation}
			
	This observation implies the following.
		
		\begin{observation}
		\label{obs2}
			For odd $n$, 
			let $P$ be a Latin square of order $n$ and let $L$ be the Smetaniuk
		completion of $T^2(P)$ with augmented forward diagonal $D^2$.
			If $$\{P(1,2), P(1,3)\} \cap 
			\{P(2,4), P(2,5), P(3,4), P(3,5)\} = \emptyset,$$
			and the set $\{((2,4)_P, (2,5)_P, (3,4)_P, (3,5)_P\}$
			is an intercalate, then the set \newline
			$\{(1,4)_L, (1,5)_L, (2,4)_L, (2,5)_L\}$ is an intercalate
			on the same symbols.
		\end{observation}

	\subsection{Reducing partial Latin squares}

	Kuhl and McGinn \cite{KuhlMcGinn} decribed a method for ``reducing''
	elements of $\mathrm{PLS}(a,b;n)$.  We sketch their
	method below; for a more elaborate exposition, see \cite{KuhlMcGinn}.
	
	Let $a,b,j,k \in [n]$, let $P \in \mathrm{PLS}(a,b;n)$ and denote
	by $C_j$ and $R_k$ column $j$ and row $k$, respectively, as subarrays of $P$.
	As for partial Latin squares, 
	we shall often treat these subarrays as sets of ordered triples,
	i.e.
	$C_j = \{(i,j,s) : (i,j,s) \in P, i \in [n], s \in [n]\}$, and similarly for rows.
	
	Henceforth, for $P \in \mathrm{PLS}(a,b;n)$, we shall assume that all nonempty
	cells of $P$ are in the first $a$ rows and  first $b$ columns of $P$.
	For any two columns $C_j$ and $C_k$ in $P$, 
	we define the {\em column composition} $C_j \circ_l C_k$, where $l \leq a$,
	as a new column with the same elements as $C_j$ except that the symbol in
	row $l$ of  $C_j \circ_l C_k$ is $P(l,k)$. A {\em row composition}
	is defined as a column composition in the row-column conjugate $P^{(rc)}$
	of $P$.
	
	Now, let $P \in \mathrm{PLS}(2,3;n)$ and
	assume $\alpha$ is a symbol
	not occurring in the $2 \times 3$ subarray  in the upper left corner of $P$;
	assume further that
	$$P(j,1) = P(k,2) = P(l,3) = P(1,q)= P(2,r) = \alpha.$$
	If there is an $i \in [n] \setminus [2]$, such that $R_j \circ_1 R_i$,
	$R_k \circ_2 R_i$, $R_l \circ_3 R_i$ are Latin (i.e. contains no repeated
	symbols), then we say that 
	$\alpha$ is a {\em row-replacable symbol} and that row $R_i$ {\em replaces}
	$\alpha$. If $i \in \{j,k,l\}$, then $R_i$ {\em replaces itself}.
	Similarly, if there is a $p \in [n]\setminus [3]$ such that
	$C_q \circ_1 C_p$ and $C_r \circ_2 C_p$ are Latin, then $\alpha$ is a
	{\em column-replacable symbol}, and $C_p$ {\em replaces} $\alpha$.
	If $p \in \{q,r\}$, then $C_p$ {\em replaces itself}.
	
	If $\alpha$ is both row- and column-replacable, then we say that
	$\alpha$ is {\em replacable}. If $\alpha$ is replacable with $R_i$
	and $C_p$ replacing $\alpha$ as above, then we define the {\em reduction}
	of $A$, denoted $R(P; R_i, C_p, \alpha)$ as the array obtained
	by removing rows $R_j, R_k, R_l$ and columns $C_q, C_r$ from $P$,
	and adding the rows $R_j \circ_1 R_i$, $R_k \circ_2 R_i$, $R_l \circ_3 R_i$,
	and columns $C_q \circ_1 C_p$, $C_r \circ_2 C_p$, and then finally
	removing $C_p$ and $R_i$ from $P$.
	Note that, for the purpose of completability, we may by isotopy
	assume that $R(P;R_i, C_p, \alpha)$
	is a partial Latin square; that is, the removed symbol is $n$
	and the last column and row are removed when forming
	$R(P;R_i, C_p, \alpha)$.
	
	The following was proved in \cite{KuhlMcGinn}.
	
	\begin{lemma}
	\label{lem:reduce}
	\cite{KuhlMcGinn}
		Let $P \in \mathrm{PLS}(2,3;n)$ where $n \geq 9$.
		If $\alpha$ is a symbol that does not occur in the
		intersection of the filled rows and columns of $P$,
		then there is a row replacing $\alpha$.	
	\end{lemma}
		
	A partial Latin square $P \in \mathrm{PLS}(2,3;n)$
	is {\em reducible} if there is a symbol $\alpha$,
	a row $R_i$ and a column $C_j$, such that row $R_i$ replaces
	$\alpha$ and column $C_j$ replaces $\alpha$ and itself;
	we say that the reduction $R(P; R_i, C_j, \alpha)$ is a {\em proper reduction}
	of $P$. For a sequence of partial Latin squares $A_1, A_2, \dots, A_m$,
	where $A_{i+1}$ is a proper reduction of $A_i$, $i=1,\dots,m-1$,
	we say that $A_m$ is obtained by {\em successive reduction}
	of $A_1$ and that $A_1$ can be {\em successively reduced} to $A_m$.
	
	The following is a main result of the method in \cite{KuhlMcGinn};
	here formulated for partial Latin squares in $\mathrm{PLS}(2,3;n)$.
	
	\begin{theorem}
	\label{th:propreduction}
	\cite{KuhlMcGinn}
		If $P \in \mathrm{PLS}(2,3;n)$ is reducible
		and one of its proper reductions is completable, then $P$
		is completable.
	\end{theorem}

	\section{Completing partial latin squares in $\mathrm{PLS}(2,3;n)$}
	
	In this section we describe our method for completing
	partial Latin squares in $\mathrm{PLS}(2,3;n)$. Throughout the rest of the paper,
	we assume
	that every partial Latin square from this family has
	all nonempty cells in the first two rows and three first columns.
	
	\subsection{Reducibility}
	
	If $P$ is an $n \times n$ partial latin square where rows $r_1$ and $r_2$
	are completely filled, then the {\em $(r_1,r_2)$-row-permutation}
	of $P$ is the permutation $\sigma : [n] \to [n]$ defined by
	$\sigma(P(r_1,i)) = P(r_2,i)$ for every $i \in [n]$.

	Consider 
	the disjoint cycle representation of
	a row-permutation $\sigma$ of $P \in \mathrm{PLS}(2,b;n)$.
	A {\em cycle type} of a cycle $C$ of length $m$ 
	in this representation
	of $\sigma$ is a sequence $s$ of $m$ integers, where the $i$th element
	of $s$ is $1$ if the $i$th element of $C$ appears in the upper left
	$1 \times b$ subarray of $P$; and $0$ otherwise.
	Two cycle types are {\em equivalent} if one of them can be obtained
	from the other by permuting the elements in the sequence cyclically.
	If a cycle in the disjoint cycle representation of $\sigma$ has a cycle
	type that is equivalent to $s$, then $s$ {\em occurs in $\sigma$}.
	
	For all non-equivalent cycle types $s$ that occurs in the permutation
	$\sigma$, let $i_s$ be the number of cycles of (the disjoint cycle
	representation of) $\sigma$
	that have a cycle type
	that is equivalent to $s$. The set of all ordered pairs $(s,i_s)$,
	where $s$ is a cycle type that occurs in $\sigma$, is called
	the {\em cycle type of $\sigma$}, or the {\em cycle type of $P$}
	if $P \in \text{PLS}(2,3;n)$ and $\sigma$ is the $(1,2)$-row-permutation of $P$.
	
	Two cycle types $A_1$ and $A_2$ of row permutations are
	{\em equivalent} if they correspond to two different disjoint cycle
	representations of the same permutation. Note that $A_1$
	and $A_2$ are equivalent if and only if there is a bijection 
	$\varphi : A_1 \to A_2$ such that $\varphi((s,i_s)) = (t, i_t)$
	if and only if $s$ and $t$ are equivalent and $i_s = i_t$. 
	
	Two elements in a sequence $s$ are called {\em adjacent} if one
	is immediately followed by the other.

	\begin{definition}
	Let $P \in \mathrm{PLS}(2,3;n)$ be a partial Latin square, where
	$n \geq 8$. $P$ is {\em completely reduced} if the cycle type
	of every cycle in the
	disjoint cycle representation of the $(1,2)$-row-permutation of $P$
	is equivalent to one of the following sequences:

	\begin{itemize}
	
		\item[(i)] 00,
		
		\item[(ii)] 01,
		
		\item[(iii)] 11,
		
		\item[(iv)] 101,
		
		\item[(v)] 111,
		
		\item[(vi)] 1010,
		
		\item[(vii)] 1110,
		
		\item[(viii)] 10101,
		
		\item[(ix)] 101010,
	
	\end{itemize}
	\end{definition}
	
	If $P'$ is a completely reduced partial Latin square that is obtained
	from successive reduction of $P$, then $P'$ is called a 
	{\em complete reduction} of $P$.
	
	We shall use the following simple observation.
	
	\begin{lemma}
		Let $P \in \mathrm{PLS}(2,3;n)$, where $n \geq 8$.
		If $P$ is not completely reduced, then the cycle type of at least
		one of the cycles of length at least $3$ 
		in the $(1,2)$-row-permutation of $P$ contains
		two adjacent zeros.	
	\end{lemma}

	The following theorem is now easy to prove.
	
	\begin{theorem}
	\label{th:reduction}
		A partial Latin square $P \in \mathrm{PLS}(2,3;n)$, where $n \geq 9$,
		has a proper reduction if and only if it is not completely reduced.	
	\end{theorem}
	\begin{proof}
		If $P$ is not completely reduced, then by the preceding lemma,
		there is a cycle $C$ of length at least $3$
		in the $(1,2)$-row-permutation $\sigma$ of $P$
		whose cycle type contains two
		adjacent zeros. This means that $C$
		contains a symbol $s$ that is neither
		contained in the $2 \times 3$ subarray in the upper left corner of $P$,
		nor in an intercalate contained in the first two rows of $P$.
		Hence, the two columns containing $s$ each replace $s$ and themselves,
		respectively. Moreover, by Lemma \ref{lem:reduce}, there is a row
		replacing $s$. Hence, $P$ has a proper reduction.
		
		Conversely, if there is a proper reduction of $P$, then, since
		there is a column replacing itself, there is a
		cycle of length at least $3$ in $\sigma$ that has a cycle type with 
		adjacent zeros. Thus $P$ is not completely reduced.
	\end{proof}
	
	It follows from this theorem that from any partial Latin square
	in $\mathrm{PLS}(2,3;n)$, $n \geq 9$, we can by succesive reduction
	obtain a partial Latin square in $\mathrm{PLS}(2,3;8)$,
	or a partial Latin square in $\mathrm{PLS}(2,3;m)$ with a cycle type that is
	equivalent to one of the following cycle types:
	
	\begin{itemize}
	
		\item[(a)] $\{(10,3), (00,k)\}$,
		
		\item[(b)] $\{(10,1), (11,1), (00,k+1)\}$,
	
		\item[(c)] $\{(10,1), (101,1), (00,k+1)\}$,
		
		\item[(d)] $\{(10,1), (1010,1), (00,k)\}$,
		
		\item[(e)] $\{(111,1), (00,k+2)\}$,
	
		\item[(f)] $\{(1110,1), (00,k+1)\}$,
		
		\item[(g)] $\{(10101,1), (00,k+1)\}$,
		
		\item[(h)] $\{(101010,1), (00,k)\}$,
		
	\end{itemize}
	where $k \geq 1$.
	
	Thus, for proving Conjecture \ref{conj:general} it suffices to show
	that all partial Latin squares in $\mathrm{PLS}(2,3;8)$ as well
	as all of type (a)-(h) can be completed.
	In the next section we shall verify the former statement and
	also prove that all partial Latin squares of the type (e) have
	completions.
	
	\subsection{Completing a particular family in $\mathrm{PLS}(2,3;n)$}

	In this section we prove that all partial Latin squares in 
	$\mathrm{PLS}(2,3;n)$ with a specific cycle type are completable.

	\begin{theorem}
	\label{th:PLS}
		If $P \in \mathrm{PLS}(2,3;n)$ is a partial Latin square
		with cycle type $\{(111,1), (00, k+2)\}$, $k \geq 3$,
		then $P$ is completable.
	\end{theorem}
	\begin{proof}
		In the proof we shall, by slight abuse of terminology,
		for simplicity allow partial Latin squares of
		order $m$ to have a different symbol set than $\{1,\dots,m\}$.
		
		Let $P$ be a partial Latin square satisfying the conditions
		in the theorem; so $n =3+2(k+2)$.
		In particular, $P$ has odd order, since the $(1,2)$-row-permutation
		of $P$ contains one cycle of length $3$ and $k+2$ cycles of length $2$.
		Moreover, by isotopy, we may assume that
		\begin{itemize}
		
		\item $P(1,i)=P(i,1)=i, i=1,\dots,n$,
		
		\item  $P(2,2)=3, P(2,3) =1$, and
		
		\item
		$P(2, 2i) = 2i+1$ and $P(2,2i+1)=2i$ for $i=2,\dots, \frac{n-1}{2}$.
		\end{itemize}
		
		Consider the row-symbol conjugate $P^{(rs)}$ of $P$.
		Since the two first rows of $P$ are completely filled,
		the augmented forward diagonal of $P^{(rs)}$ is completely filled
		with the symbols $1$ and $2$, and
		moreover
		
		\begin{itemize}
			
			\item[(i)] $P(1,1)=1, P(1,2) =s_1, P(1,3) = 2$,
			
			\item[(ii)] $P(2,1)=2, P(2,2)=1, P(2,3)=s_2$, and
			
			\item[(iii)] $P(3,i)=4-i$, for $i=1,2,3$,
		
		\end{itemize}
		where $s_1$ and $s_2$ are some symbols from $\{1,\dots,n\}$.
		
		Now, if $s_1 = s_2 = 3$, then $P$ has a completion by 
		Theorem \ref{th:Haggkvist}. Thus it suffices to consider the
		following cases:
		
		\begin{itemize}
		
		\item[(a)] $s_1$ and $s_2$ are distinct, and $3 \notin \{s_1,s_2\}$,
		
		\item[(b)] $s_1$ and $s_2$ are distinct, and $3 \in \{s_1,s_2\}$,
		
		\item[(c)] $s_1 = s_2 \neq 3$.
		
		\end{itemize}
		
		Suppose first that (a) holds, and assume without loss of generality
		that $s_1 =4$ and $s_2 =5$. We define the partial Latin square
		$C$ of order $n-2$ by letting $(1,1,3), (1,2,4), (1,3,5) \in C$,
		and
		for $i=2,\dots,n-2$
		letting $(i,j,k) \in C$ if and only if $(i+2,j,k) \in P$
		and $(i+2,j)_P$ is not contained in the augmented forward diagonal
		of $P$. Then $C$ is a partial Latin square on the symbols
		$[n] \setminus [2]$, which by Theorem \ref{th:MHall} is completable.
		It thus follows from Theorem \ref{th:T2} that $T^2(C)$ has a
		Smetaniuk completion $A$, where symbols $1,2$ appear in the augmented
		forward diagonal. By possibly making some swaps on intercalates
		with symbols $1$ and $2$ in $A$, we obtain a completion of $P^{(rs)}$;
		so, by conjugacy, $P$ is completable.
		
		Suppose now that (b) holds. By isotopy, we may assume 
		$\{s_1,s_2\} = \{3,4\}$. Suppose e.g. that $s_1=3$ and $s_2=4$
		(the other case is similar).
		With this assumption, it is straightforward that by
		permuting the first three rows and columns, and symbols $3$ and $4$,
		we can obtain, from $P^{(rs)}$ a partial Latin square that satisfies
		conditions (i)-(iii) and (c). 
		We conclude that is suffices to consider the case when
		(c) holds.

		So assume that (c) holds. Without loss of generality,
		we assume that $s_1 = s_2 =4$. From $P^{(rs)}$, we shall define
		a sequence of partial Latin squares; each partial Latin square
		will contain an isotopism of the previous one, or will be obtained from it
		by a swap on an intercalate.
		We first define a partial Latin square $B_1$ by applying the permutation
		$(4 \,\, 6) (5 \,\, 7)$ to the rows and columns of $P^{(rs)}$ if
		$(i,1,4) \in P^{(rs)}$ for some $i \in \{4,5\}$; otherwise
		we set $B_1 = P^{(rs)}$.
		
		Next, we put $S= \{B_1(4,i), B_1(5,i) : i \in [3] \}$
		and pick a row $q \geq 6$ in $B_1$ such that
		$|\{B_1(q,2), B_1(q,3)\} \cap (S \cup \{3\})| \leq 1$
		and $B_1(q,1) \neq 4$; since $B_1$ has order at least $13$,
		there is such a row $q$. We set $B_2 = B_1 \cup \{q,4,4\}$.
		
		We define the partial Latin square $B_3$ by permuting the rows
		and columns in $B_2$ according to $(1 \,\, 2)$ and $(2 \,\, 3)$, respectively,
		if $B_2(q,3) \in S\cup \{3\}$; otherwise we set $B_3 =B_2$.
		Then $B_3(q,3) \notin S\cup \{3\}$. We put 
		$B_4 =B_3 \cup \{(2,4,B_3(q,3))\}$ and note that
		the set $\{(q,3)_{B_4}, (q,4)_{B_4}, (2,3)_{B_4}, (2,4)_{B_4}\}$
		is an intercalate in $B_4$. We swap on this intercalate to obtain $B_5$.
		
		Next, we pick a symbol $\alpha \notin S \cup \{1,2,3,4, B_5(2,3)\}$;
		since $n \geq 13$ and $$|S \cup \{1,2,3,4,B_5(2,3)\}| \leq 11,$$
		there is indeed such a symbol $\alpha$.
		We set $B_6 = B_5 \cup \{(2,5,\alpha), (3,4,\alpha), (3,5,4)\}$ and note
		that the cells 
		$$\{(2,4)_{B_6}, (2,5)_{B_6}, (3,4)_{B_6}, (3,5)_{B_6}\}$$
		form an intercalate in $B_6$.
		We now permute the rows and columns according to $(1 \, \, 3)$
		and $(1 \,\, 2)$ in $B_6$, respectively, 
		and denote the obtained partial Latin square
		by $B_7$.
		
		From $B_7$ we  define a partial Latin square $C$ of order $n-2$
		by letting $$(1,1,4), (1,2,3), (1,3, B_7(2,3)), (2,4, \alpha), (3,4,4) \in C,$$
		and for $i=2,\dots,n-2$ letting $(i,j,s) \in C$ if and only if
		$(i+2,j,k) \in B_7$ 
		and $(i+2,j)_{B_7}$ is not contained in the augmented forward diagonal
		of $B_7$. 
		Since $\{4,\alpha\} \cap S = \emptyset$,
		$C$ is a partial Latin square on the symbols
		$[n] \setminus [2]$.
		
		Now, since column $4$ in $C$ contains three nonempty cells,		
		Lemma \ref{lem:PLS} implies that $C$ has a completion $C'$.
		We define $C_1$ from $C$ by filling column $4$ in $C$
		as in column $4$ in $C'$, and, in addition, setting
		$C_1(2,5)=4$ and $C_1(3,5)=\alpha$. Then, by construction,
		$C_1$ satisfies the following:
		
		\begin{itemize}
		
			\item row $2$ in $C_1$ contains symbols 
			$B_4(4,1)$, $B_4(4,2)$, $B_4(4,3)$, $\alpha$, $4$,
			which are all distinct by the choice of $\alpha$ and the construction of $B_4$;
			
			\item row $3$ in $C_1$ contains symbols
			$B_4(5,1)$, $B_4(5,2)$, $B_4(5,3)$, $\alpha$, $4$
			which are all distinct;
			
			\item for $i=4,\dots, n-2$, row $i$ of $C_1$
			contain the same symbols as row $i+2$ in $B_4$.
			
		\end{itemize}
		Hence, $C_1$ is a partial Latin square over the symbols
		$[n] \setminus [2]$.
		
		Again, by Lemma \ref{lem:PLS}, there is a completion $C'_1$ of
		$C_1$. Moreover, since $C_1$ is completable, it follows from
		Theorem \ref{th:T2} that $T^2(C_1)$ has a completion $A$ with 
		an augmented forward diagonal with the symbols $1$ and $2$;
		we choose this augmented forward diagonal of
		$A$ so that it agrees with the augmented forward diagonal of $B_7$.
		Hence, the first three columns of $A$ agree with $B_7$.
		Moreover, by Observations \ref{obs1}-\ref{obs2}, 
		the cells in positions $(1,4), (1,5), (2,4), (2,5)$ in $A$
		form an intercalate $F$ on the symbols $\{4,\alpha\}$, as in $B_7$,
		and $A(q,4)= B_7(q,4) = B_4(q,3)$.
		
		Now, if $A(2,4)=\alpha$, then we swap on the intercalate $F$ to obtain
		the partial Latin square $A'$; otherwise if $A(2,4)=4$,
		we set $A'=A$. In $A'$, the set
		$$\{(2,3)_{A'}, (2,4)_{A'}, (q,3)_{A'},(q,4)_{A'}\}$$
		is an intercalate, and by swapping on this intercalate we
		obtain a completion of an isotopism of $B_4$. Now, since an isotopism of
		$P^{(rs)}$ is contained in $B_4$, $P^{(rs)}$ is completable,
		and so,
		$P$ is completable.		
	\end{proof}
	
		Consider a partial Latin square $P$ in $\mathrm{PLS}(2,3;n)$, where 
		$n \geq 8$ and the $2 \times 3$ subarray in the upper left corner
		forms a Latin rectangle.
		By Theorem \ref{th:reduction} one of the following must hold.
		
		\begin{itemize}
		
			\item[(i)] There is a complete reduction $P'$ of $P$ with odd order
			at least $13$.
			
			\item[(ii)] There is a complete reduction $P'$ of $P$ with odd order
			$11$.
			
			\item[(iii)] There is a complete reduction $P'$ of $P$ with odd
			order $9$.
			
			\item[(iv)] $P$ or a partial Latin square obtained from $P$ by successive
			reduction is a partial Latin square of order $8$.
		
		\end{itemize} 
		
		Theorems \ref{th:PLS} and \ref{th:propreduction} 
		implies that every partial Latin square satisfying
		(i) is completable.
		The cases (ii)-(iv) have been settled by a computer search; as it turns
		out, every partial Latin square 
		in $\mathrm{PLS}(2,3;n)$ with cycle type $\{(111,1), (00, k)\}$,
		where $n=3+2k$ and
		$k \in\{3,4\}$, is completable; additionally, this holds for 
		all partial Latin squares
		in $\text{PLS}(2,3;8)$ as well;
		for details, see \cite{Goransson}.
		Hence, by Theorems \ref{th:propreduction} and
		\ref{th:PLS} we have the following.
		
		\begin{corollary}
		\label{cor:main}
			Every partial Latin square of order at least $8$ with
			two filled rows and three filled columns,
			and where the intersection of the filled rows and columns form a Latin 
			rectangle, is completable.
		\end{corollary}

\end{document}